\documentclass[a4paper,10pt]{article}

\usepackage{geometry}
\usepackage{amsmath}
\usepackage{amssymb}
\usepackage{amsthm}
\usepackage{mathrsfs}
\usepackage{color}
\usepackage{amscd}
\usepackage{graphicx}

\usepackage{amsthm}
\newtheorem{theorem}{Theorem}
\newtheorem{definition}[theorem]{Definition}

\newtheorem{proposition}[theorem]{Proposition}
\newtheorem{corollary}[theorem]{Corollary}
\newtheorem{example}[theorem]{Example}

\theoremstyle{definition}
\newtheorem{remark}[theorem]{Remark}

\usepackage{amsfonts}

\newcommand{\hh}{{\mathbb{H}}}

\newcommand{\cc}{{\mathbb{C}}}
\newcommand{\rr}{{\mathbb{R}}}

\newcommand{\s}{{\mathbb{S}}}

\newcommand{\punto}{\cdot}

\newcommand\re{\operatorname{Re}}
\newcommand\im{\operatorname{Im}}

\newcommand{\mk}{\mathfrak}

\newcommand{\ch}{\mathcal K}
\newcommand{\sn}{\mk{sn}}

\title{\bf The quaternionic Gauss-Lucas Theorem} 

\author{Riccardo Ghiloni\\
 Alessandro Perotti\\
\small Department of Mathematics, University of Trento\\ 
\small Via Sommarive 14, I-38123 Povo Trento, Italy\\
\small riccardo.ghiloni@unitn.it, alessandro.perotti@unitn.it}

\date{  }

\begin{document}

\maketitle


\begin{abstract}
The classic Gauss-Lucas Theorem for complex polynomials of degree $d\ge2$ has a natural reformulation over quaternions, obtained via rotation around the real axis. We prove that such a reformulation is true only for $d=2$. We present a new quaternionic version of the Gauss-Lucas Theorem valid for all $d\geq2$, together with some consequences.
\end{abstract}


\section{Introduction}

Let $p$ be a complex polynomial of degree $d\geq2$ and let $p'$ be its derivative. The Gauss-Lucas Theorem asserts that the zero set of $p'$ is contained in the convex hull $\ch(p)$ of the zero set of $p$. The classic proof uses the logarithmic derivative of $p$ and it strongly depends on the commutativity of $\cc$.

This note deals with a quaternionic version of such a classic result. 
We refer the reader to \cite{GeStoSt2013} for the notions and properties concerning the algebra $\hh$ of quaternions we need here.
The ring $\hh[X]$ of quaternionic polynomials is defined by fixing the position of the coefficients w.r.t.\  the indeterminate $X$ (e.g.\ on the right) and by imposing commutativity of $X$ with the coefficients when two polynomials are multiplied together (see e.g.\ \cite[\S 16]{Lam}). Given two polynomials $P,Q\in\hh[X]$, let $P\punto Q$ denote the product obtained in this way. If $P$ has real coefficients, then $(P\punto Q)(x)=P(x)Q(x)$.
In general, a direct computation (see \cite[\S 16.3]{Lam}) shows that if $P(x)\ne0$, then
\begin{equation}\label{product}
(P\punto Q)(x)=P(x)Q(P(x)^{-1}xP(x)),
\end{equation}
while $(P\punto Q)(x)=0$ if $P(x)=0$.
In this setting, a {(left) root} of a polynomial $P(X)=\sum_{h=0}^dX^h a_h$ is an element $x\in\hh$ such that $P(x)=\textstyle\sum_{h=0}^dx^h a_h=0$.

Given  $P(X)=\sum_{k=0}^dX^ka_k\in \hh[X]$, 
consider the polynomial $P^c(X)=\sum_{k=0}^dX^k\bar a_k$ and the \emph{normal polynomial} $N(P)=P \cdot P^c=P^c\cdot P$. Since $N(P)$ has real coefficients, it can be identified with a polynomial in $\rr[X]\subset\cc[X]$. We recall that a subset $A$ of $\hh$ is called \emph{circular} if, for each $x\in A$, $A$ contains the whole set (a 2-sphere if $x\not\in\rr$, a point if $x\in\rr$)
\begin{equation}\label{sx}
\s_x=\{pxp^{-1}\in\hh\;|\;p\in\hh^*\},
\end{equation}
where $\hh^*:=\hh\setminus\{0\}$. In particular, for any imaginary unit $I\in\hh$, $\s_I=\s$ is the 2-sphere of all imaginary units in $\hh$. For every subset $B$ of $\hh$ we define its \emph{circularization} as the set $\bigcup_{x\in B}\s_x$. It is well-known (\cite[\S3.3]{GeStoSt2013}) that the zero set $V(N(P))\subset\hh$ of the normal polynomial is the circularization of the zero set $V(P)$, which consists of isolated points or isolated 2-spheres of the form \eqref{sx} if $P\neq0$.

Let the degree $d$ of $P$ be at least 2 and let $P'(X)=\sum_{k=1}^dX^{k-1}ka_k$ be the derivative of $P$. 
It is known (see e.g.\ \cite{GGSarxiv}) that the Gauss-Lucas Theorem does not hold directly for quaternionic polynomials. For example, the polynomial $P(X)=(X-i)\punto (X-j)=X^2-X(i+j)+k$ has zero set $V(P)=\{i\}$, while $P'$ vanishes at $x=(i+j)/2$.

Since the zero set $V(P)
$ of $P$ is contained in the set $V(N(P))
$, a natural reformulation in $\hh[X]$ of the classic Gauss-Lucas Theorem is the following: $V(N(P'))\subset\ch(N(P))$ or equivalently
\begin{equation}\label{eq:g-l}
V(P')\subset\ch(N(P)),
\end{equation}
where $\ch(N(P))$ denotes the convex hull of $V(N(P))$ in $\hh$. This set is equal to the circularization of the convex hull of the zero set of $N(P)$ viewed as a polynomial in $\cc[X]\subset\hh[X]$. Recently two proofs of the above inclusion \eqref{eq:g-l} were presented in \cite{VlacciGL,GGSarxiv}.

Our next two propositions prove that inclusion \eqref{eq:g-l} is correct in its full generality only when $d=2$. 

\section{Gauss-Lucas polynomials}

\begin{definition}
Given a polynomial $P\in\hh[X]$ of degree $d\geq 2$ we say that $P$ is a \emph{Gauss-Lucas polynomial} if $P$ satisfies \eqref{eq:g-l}. 
\end{definition}

\begin{proposition}
If $P$ is a polynomial in $\hh[X]$ of degree $2$, then $V(N(P))=\s_{x_1}\cup\s_{x_2}$ for some $x_1,x_2\in\hh$ (possibly with $\s_{x_1}=\s_{x_2}$) and
\[
V(P')\subset\bigcup_{y_1\in\s_{x_1},y_2\in\s_{x_2}}\left\{\frac{y_1+y_2}{2}\right\}.
\]
In particular every polynomial $P\in\hh[X]$ of degree $2$ is a Gauss-Lucas polynomial.
\end{proposition}
\begin{proof}
Let $P(X)=X^2a_2+Xa_1+a_0\in\hh[X]$ with $a_2\neq0$. Since $P\cdot a_2^{-1}=Pa_2^{-1}$, $(P\cdot a_2^{-1})'=P'\cdot a_2^{-1}=P'a_2^{-1}$, we can assume $a_2=1$. Consequently, $P(X)=(X-x_1)\cdot(X-x_2)=X^2-X(x_1+x_2)+x_1x_2$ for some $x_1,x_2\in\hh$. Then $x_1\in V(P)$ and $\bar x_2\in V(P^c)$, since $P^c(X)=(X-\bar x_2)\cdot(X-\bar x_1)$. Therefore $x_1,x_2\in V(N(P))$. On the other hand $V(P')=\{(x_1+x_2)/2\}$ as desired.
\end{proof}

\begin{remark}
Let $P(X)=\sum_{k=0}^dX^ka_k\in\hh[X]$ of degree $d\geq2$ and let $Q:=P\cdot a_d^{-1}$ be the corresponding monic polynomial. Since $V(P)=V(Q)$ and $V(P')=V(Q')$, $P$ is a Gauss-Lucas polynomial if and only if $Q$ is.
\end{remark}

\begin{proposition} \label{prop:-4}
Let $P\in\hh[X]$ of degree $d\geq3$. Suppose that $N(P)(X)=X^{2e}\cdot(X^2+1)^{d-e}$ for some $e<d$ and that $N(P')(X)=\sum_{k=0}^{2d-2}X^kb_k$ contains a unique monomial of odd degree, that is, $b_k\neq0$ for a unique odd $k$. Then $P$ is not a Gauss-Lucas polynomial.
\end{proposition}
\begin{proof}
Since $V(N(P))\subset\{0\}\cup\s$, $\ch(N(P))\subset\im(\hh)=\{x\in\hh\,|\,\re(x)=0\}$. Then it suffices to show that $N(P')$ has at least one root in $\hh\setminus\im(\hh)$. Let $G,L\in\rr[X]$ be the unique real polynomials such that $G(t)+iL(t)=N(P')(it)$ for every $t\in\rr$. Since $N(P')$ contains a unique monomial of odd degree, say $X^{2\ell+1}b_{2\ell+1}$, then $L(t)=(-1)^\ell b_{2\ell+1}t^{2\ell+1}$ and hence $V(N(P'))\cap i\rr\subset\{0\}$. Being $V(N(P'))$ a circular set, it holds $V(N(P'))\cap\im(\hh)\subset\{0\}$. 
Since $N(P')$ contains at least two monomials, namely those of degrees $2\ell+1$ and $2d-2$, we infer that it must have a nonzero root. Therefore 
$V(N(P'))\not\subset\{0\}$ and hence $V(N(P'))\not\subset\im(\hh)$, as desired.

\end{proof}

\begin{corollary}\label{counterexample}
Let $d\geq3$ and let
\[
P(X)=X^{d-3}\cdot(X-i)\cdot(X-j)\cdot(X-k).
\]
Then $N(P)(X)=X^{2d-6}\cdot(X^2+1)^3$ and $N(P')$ contains a unique monomial of odd degree, namely $-4X^{2d-5}$. In particular $P$ is not a Gauss-Lucas polynomial.
\end{corollary}
\begin{proof}
By a direct computation we obtain:
\begin{align}
P(X)&=X^d-X^{d-1}(i+j+k)+X^{d-2}(i-j+k)+X^{d-3},\\
P'(X)&=dX^{d-1}-(d-1)X^{d-2}(i+j+k)+(d-2)X^{d-3}(i-j+k)+(d-3)X^{d-4},\label{P1}
\\
N(P')(X)&=d^2X^{2d-2}+3(d-1)^2X^{2d-4}-4X^{2d-5}+3(d-2)^2X^{2d-6}+(d-3)^2X^{2d-8}.
\end{align}
Proposition \ref{prop:-4} implies the thesis.
\end{proof}

Let $I\in\s$ and let $\cc_I\subset\hh$ be the complex plane generated by 1 and $I$. Given a polynomial $P\in\hh[X]$, we will denote by $P_I:\cc_I\to\hh$ the restriction of $P$ to $\cc_I$. {\it If $P_I$ is not constant, we will denote by $\ch_{\cc_I}(P)$ the convex hull in the complex plane $\cc_I$ of the zero set $V(P_I)=V(P)\cap\cc_I$. If $P_I$ is constant, we set $\ch_{\cc_I}(P)=\cc_I$.}

If $\cc_I$ contains every coefficient of $P\in\hh[X]$, then we say that $P$ is a \emph{$\cc_I$-polynomial}.

\begin{proposition}
The following holds:
\begin{enumerate}
 \item[(1)] Every $\cc_I$-polynomial of degree $\geq2$ is a Gauss-Lucas polynomial.
 \item[(2)] Let $d\geq3$, let $\hh_d[X]=\{P\in\hh[X]\,|\,\deg(P)=d\}$ and let $E_d[X]$ be the set of all elements of $\hh_d[X]$ that are not Gauss-Lucas polynomials.
 Identify each $P(X)=\sum_{k=0}^dX^ka_k$ in $\hh_d[X]$ with $(a_0,\ldots,a_d)\in\hh^d\times\hh^*\subset\rr^{4d+4}$ and endow $\hh_d[X]$ with the relative Euclidean topology. 
 Then $E_d[X]$ is a nonempty open subset of $\hh_d[X]$. Moreover $E_d[X]$ is not dense in $\hh_d[X]$, being $X^d-1$ an interior point of its complement. 
\end{enumerate}
\end{proposition}
\begin{proof}
If $P$ is a $\cc_I$-polynomial, then $P_I$ can be identified with an element of $\cc_I[X]$. Consequently, the classic Gauss-Lucas Theorem gives $V(P')\cap\cc_I=V(P_I')\subset\ch_{\cc_I}(P)$. 
The zero set of the $\cc_I$-polynomial $P'$ has a particular structure (see \cite[Lemma~3.2]{GeStoSt2013}): $V(P')$ is the union of $V(P')\cap\cc_I$ with the set of spheres $\s_x$ such that $x,\bar x\in V(P'_I)$. It follows that
\[
V(P')\subset\ch(N(P)).
\]

This proves (1).

Now we prove (2). By Corollary \ref{counterexample} we know that $E_d[X]\neq\emptyset$. If $P\in E_d[X]$, then $V(N(P'))\not\subset\ch(N(P))$. 
$N(P)$ and $N(P')$ 
are polynomials with real coefficients. Since the roots of $N(P)$ and of $N(P')$ depend continuously on the coefficients of $P$ and $\ch(N(P))$ is closed in $\hh$, for every $Q\in\hh_d[X]$ sufficiently close to $P$, $V(N(Q'))$ is not contained in $\ch(N(Q))$, that is $Q\in E_d[X]$.

To prove the last statement, observe that $P(X)=X^d-1$ is not in $E_d[X]$ from part (1). Since $V(P')=V(N(P'))=\{0\}$ is contained in the interior of the set $\ch(N(P))$, for every $Q\in\hh_d[X]$ sufficiently close to $P$, $V(Q')$ is still contained in $\ch(N(Q))$.
\end{proof}


\section{A quaternionic Gauss-Lucas Theorem}

Let $P(X)=\sum_{k=0}^dX^ka_k\in\hh[X]$ of degree $d\geq2$. For every $I\in\s$, let $\pi_I:\hh\to\hh$ be the orthogonal projection onto $\cc_I$ and $\pi_I^\bot=id-\pi_I$. Let $P^I(X):=\sum_{k=1}^dX^ka_{k,I}$ be the $\cc_I$-polynomial with coefficients $a_{k,I}:=\pi_I(a_k)$. 

\begin{definition}
We define the \emph{Gauss-Lucas snail of $P$} as the following subset $\sn(P)$ of $\hh$:
\[
\sn(P):=\bigcup_{I\in\s}\ch_{\cc_I}(P^I).
\]
\end{definition}

Our quaternionic version of the Gauss-Lucas Theorem reads as follows. 

\begin{theorem}\label{thm}
For every polynomial $P\in\hh[X]$ of degree $\geq2$, 
\begin{equation}\label{snail}
V(P')\subset\sn(P).
\end{equation}
\end{theorem}
\begin{proof}
Let $P(X)=\sum_{k=0}^dX^ka_k$ in $\hh_d[X]$ with $d\geq2$. We can decompose the restriction of $P$ to $\cc_I$ as $P_I=\pi_I\circ P_I+\pi_I^\bot\circ P_I={P^I}_{|\cc_I}+\pi_I^\bot\circ P_I$. If $x\in\cc_I$, then $P^I(x)\in\cc_I$ while $(\pi_I^\bot\circ P_I)(x)\in\cc_I^\bot$. The same decomposition holds for $P'$. This implies that $V(P')\cap\cc_I\subset V((P^I)')\cap\cc_I$. The classic Gauss-Lucas Theorem applied to $P^I$ on $\cc_I$ gives $V(P')\cap\cc_I\subset \ch_{\cc_I}(P^I)$. Since $V(P')=\bigcup_{I\in\s}(V(P')\cap\cc_I)$, the inclusion \eqref{snail} is proved.
\end{proof}

If $P$ is monic Theorem \ref{thm} has the following equivalent formulation: {\it For every monic polynomial $P\in\hh[X]$ of degree $\geq2$, it holds}
\begin{equation}
\sn(P')\subset\sn(P).
\end{equation}

\begin{remark}\label{rem:monic}
If $P$ is a nonconstant monic polynomial in $\hh[X]$, then two properties hold: 
\begin{itemize}
\item[(a)]
$\ch_{\cc_I}(P^I)$ is a compact subset of $\cc_I$ for every $I\in\s$.
\item[(b)] 
$\ch_{\cc_I}(P^I)$ depends continuously on $I$.
\end{itemize}
Let $I\in\s$. Since $P$ is monic, also $P^I$ is a monic, nonconstant polynomial and then $\ch_{\cc_I}(P^I)$ is a compact subset of $\cc_I$. This prove property (a). To see that (b) holds, one can apply the Continuity theorem for monic polynomials (see e.g.~\cite[Theorem~1.3.1]{RahmanSchmeisser}). The roots of $P^I$ depend continuously on the coefficients of $P^I$, which in turn depend continuously on $I$. Therefore the convex hull $\ch_{\cc_I}(P^I)$ depends continuously on $I$.
Observe that (a) and (b) can not hold for polynomials that are not monic. For example, let $P(X)=X^2i$. Then, given $I=\alpha_1i+\alpha_2j+\alpha_3k\in\s$,  $\ch_{\cc_I}(P^I)=\{0\}$ if $\alpha_1\ne0$ and $\ch_{\cc_I}(P^I)=\cc_I$ if $\alpha_1=0$, since in this case $P^I$ is constant.

\end{remark}

\begin{remark}\label{rem:closed}
If $P$ is a monic polynomial in $\hh[X]$ of degree $\geq2$, then its Gauss-Lucas snail is a closed subset of $\hh$. 
To prove this fact, consider $q\in\hh\setminus\sn(P)$ and choose $I\in\s$ such that $q\in\cc_I$. Write $q=\alpha+I\beta\in\cc_I$ for some $\alpha,\beta\in\rr$ and define $z:=\alpha+i\beta\in\cc$. Since $P$ is monic, 
$\sn(P)\cap\cc_I=\ch_{\cc_I}(P^I)$ is a compact subset of $\cc_I$. Moreover, 
$\ch_{\cc_I}(P^I)$ depends continuously on $I$, and then there exist an open neighborhood $U_I$ of $z$ in $\cc$ and an open neighborhood $W_I$ of $I$ in $\s$ such that the set
\[
\textstyle
[U_I,W_I]:=\bigcup_{J\in W_I}\{a+Jb\in\cc_J\;|\; a+ib\in U_I\}
\]
is an open neighborhood of $q$ in $\bigcup_{J\in W_I}\cc_J$, and it is disjoint from $\sn(P)$. If $q\not\in\rr$ then $q$ is an interior point of $\hh\setminus\sn(P)$, because $[U_I,W_I]$ is a neighborhood of $q$ in $\hh$ as well. 
Now assume that $q\in\rr$. Since $\s$ is compact there exist $I_1,\ldots,I_n\in\s$ such that $\bigcup_{\ell=1}^nW_{I_\ell}=\s$. It follows that $[\bigcap_{\ell=1}^nU_\ell,\s]$ is a neighborhood of $q$ in $\hh$, which is  disjoint from $\sn(P)$. Consequently $q$ is an interior point of $\hh\setminus\sn(P)$ also in this case. This proves that $\sn(P)$ is closed in $\hh$.

In Proposition \ref{pro:sn-compact} below we will show that the Gauss-Lucas snail of a monic polynomial in $\hh[X]$ of degree $\geq2$ is also a compact subset of $\hh$.
\end{remark}

If all the coefficients of $P$ are real, then $\sn(P)$ is a circular set. 
In general, $\sn(P)$ is neither closed nor bounded nor circular, as shown in the next example.

\begin{example}
Let $P(X)=X^2i+X$. Given $I=\alpha_1i+\alpha_2j+\alpha_3k\in\s$, $P^I(x)=X^2I\alpha_1+X$ and then $\ch_{\cc_I}(P^I)=\{0\}$ if $\alpha_1=0$ while $\ch_{\cc_I}(P^I)$ is the segment from 0 to $I{\alpha_1}^{-1}$ if $\alpha_1\ne0$. It follows that
$\sn(P)=\{x_1i+x_2j+x_3k\in\im(\hh)\,|\,0<x_1\le1\}\cup\{0\}$.
Finally, observe that the monic polynomial $Q(X)=-P(X)\cdot~i=X^2-Xi$ corresponding to $P$ has compact Gauss-Lucas snail $\sn(Q)=\{x_1i+x_2j+x_3k\in\im(\hh)\,|\,(x_1-1/2)^2+x_2^2+x_3^2\le 1/4\}$.
\end{example}

\begin{remark}
Even for $\cc_I$-polynomials, the Gauss-Lucas snail of $P$ can be strictly smaller than the circular convex hull $\ch(N(P))$. For example, consider the $\cc_i$-polynomial $P(X)=X^3+3X+2i$, with zero sets $V(P)=\{-i,2i\}$ and $V(P')=\s$. The set $\ch(N(P))$ is the closed three-dimensional disc in $\im(\hh)$, with center at the origin and radius 2. The Gauss-Lucas snail
$\sn(P)$ is the subset of $\im(\hh)$ obtained by rotating around the $i$-axis the following subset of the coordinate plane $L=\{x=x_1i+x_2j\in\im(\hh)\,|\,x_1,x_2\in\rr\}$:
\[\{x=\rho\cos(\theta)i+\rho\sin(\theta)j\in L\;|\;0\le\theta\le\pi,\,0\le\rho\le2\cos(\theta/3)\}.\] 
Therefore $\sn(P)$ is a proper subset of $\ch(N(P))$ (the boundaries of the two sets intersect only at the point $2i$). Its boundary is obtained by rotating a curve that is part of the \emph{lima\c con trisectrix} (see Figure \ref{fig:limacon}).
\end{remark}

\begin{figure}
\begin{center}
\includegraphics[width=6cm]{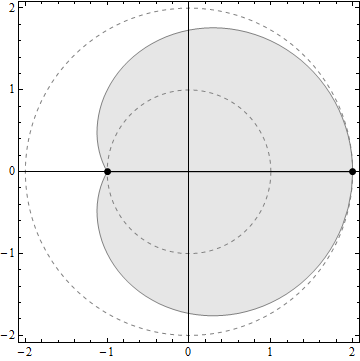}\\
\caption{Cross-sections of $\sn(P)$ (gray), of $V(P')$ and $\ch(N(P))$ (dashed).}
\label{fig:limacon}
\end{center}
\end{figure}

\subsection{Estimates on the norm of the critical points}

Let $p(z)=\sum_{k=0}^da_kz^k$ be a complex polynomial of degree $d\ge1$. The norm of the roots of $p$ can be estimated making use of the norm of the coefficients $\{a_k\}_{k=0}^d$ of $p$. There are several classic results in this direction 
(see e.g.\ \cite[\S8.1]{RahmanSchmeisser}). For instance the estimate \cite[(8.1.2)]{RahmanSchmeisser} (with $\lambda=1,p=2$) asserts that
\begin{equation}\label{eq:cauchy}
\textstyle
\max_{z\in V(p)}|z|\leq|a_d|^{-1}\sqrt{\sum_{k=0}^d|a_k|^2}\;.
\end{equation}

\begin{proposition}\label{pro:sn-compact}
For every monic polynomial $P\in\hh[X]$ of degree $d\geq2$, the Gauss-Lucas snail $\sn(P)$ is a compact subset of $\hh$.
\end{proposition}
\begin{proof}
Since $P=\sum_{k=0}^dX^ka_k$ is monic, every polynomial $P^I$ is monic. From \eqref{eq:cauchy} it follows that $\max_{x\in V(P^I)}|x|^2\le\sum_{k=0}^d|\pi_I(a_k)|^2\le\sum_{k=0}^d|a_k|^2$ and hence $\sn(P)\subset \{x\in\hh\,|\, |x|^2\le \sum_{k=0}^d|a_k|^2\}$ is bounded. Since $\sn(P)$ is closed in $\hh$, as seen in Remark \ref{rem:closed}, it is also a compact subset of $\hh$.
\end{proof}

Define a function $C:\hh[X]\to\rr\cup\{+\infty\}$ as follows: $C(a):=+\infty$ if $a$ is a quaternionic constant and
\[
\textstyle
C(P):=|a_d|^{-1}\sqrt{\sum_{k=0}^d|a_k|^2} \qquad \text{if $P(X)=\sum_{k=0}^dX^ka_k$ with $d\ge1$ and $a_d\neq0$.}
\]

\begin{proposition}\label{pro:estimate}
For every polynomial $P\in\hh[X]$ of degree $d\geq1$, it holds
\begin{equation}\label{eq:C}
\max_{x\in V(P)}|x|\leq C(P).
\end{equation}
\end{proposition}
\begin{proof}
We follow the lines of the proof of estimate \eqref{eq:cauchy} for complex polynomials given in \cite{RahmanSchmeisser}. Let $P(X)=\sum_{k=0}^dX^ka_k$ with $d\ge1$ and $a_d\neq0$. 
We can assume that $P(X)$ is not the monomial $X^d a_d$, since in this case the thesis is immediate.
Let $b_k=|a_ka_d^{-1}|$ for every $k=0,\ldots,d-1$. The real polynomial 
$h(z)=z^d-\sum_{k=0}^{d-1}b_kz^k$ has exactly one positive root $\rho$ and is positive for real $z>\rho$ (see \cite[Lemma~8.1.1]{RahmanSchmeisser}). Let $S:=\sum_{k=0}^{d-1}b_k^2=C(P)^2-1$. From the Cauchy-Schwartz inequality, it follows that
\[\left(\sum_{k=0}^{d-1}b_kC(P)^k\right)^2\le 
S\sum_{k=0}^{d-1}C(P)^{2k}=
(C(P)^2-1)\frac{C(P)^{2d}-1}{C(P)^2-1}<C(P)^{2d}.
\]
Therefore $h(C(P))>0$ and then $C(P)>\rho$.  Let $x\in V(P)$. 
It remains to prove that $|x|\le\rho$. Since $x^d=-\sum_{k=0}^{d-1}x^ka_ka_d^{-1}$, it holds
\[|x|^d\le\sum_{k=0}^{d-1}|x|^k\left|a_ka_d^{-1}\right|=\sum_{k=0}^{d-1}|x|^kb_k.\]
This means that $h(|x|)\le0$, which implies $|x|\le\rho$.
\end{proof}

From Proposition~\ref{pro:estimate} it follows that for every polynomial $P\in\hh[X]$ of degree $d\geq2$, it holds
\begin{equation}\label{eq:C}
\max_{x\in V(P')}|x|\leq C(P').
\end{equation}
Theorem \ref{thm} allows to obtain a new estimate.

\begin{proposition}
Given any polynomial $P\in\hh[X]$ of degree $d\geq2$, it holds:
\begin{equation}\label{eq:ijk'}
\max_{x\in V(P')}|x|\leq\sup_{I\in\s}\{C(P^I)\}.
\end{equation}
\end{proposition}
\begin{proof}
If $x\in V(P')\cap\cc_I$, Theorem \ref{thm} implies that $x\in\ch_{\cc_I}(P^I)$. Therefore 

\begin{equation*}\label{eq:ijk}
\max_{x\in V(P')\cap\cc_I}|x|\leq C(P^I) \quad \text{ for every $I\in\s$ with $V(P')\cap\cc_I\ne\emptyset$},
\end{equation*}
from which inequality \eqref{eq:ijk'} follows.
\end{proof}

Our estimate \eqref{eq:ijk'} can be strictly better than classic estimate \eqref{eq:C}, as explained below.

\begin{remark}
Let $d\geq3$ and let $P(X)=X^{d-3}\cdot(X-i)\cdot(X-j)\cdot(X-k)$. Using \eqref{P1}, by a direct computation we obtain:
\[
C(P')=d^{-1}\sqrt{8d^2-24d+24}.
\]
Moreover, given $I=\alpha_1i+\alpha_2j+\alpha_3k\in\s$ for some $\alpha_1,\alpha_2,\alpha_3\in\rr$ with $\alpha_1^2+\alpha_2^2+\alpha_3^2=1$, we have
$ \pi_I(i+j+k)=\langle I,i+j+k\rangle I=(\alpha_1+\alpha_2+\alpha_3)I$ and 
$\pi_I(i-j+k)=\langle I,i-j+k\rangle I=(\alpha_1-\alpha_2+\alpha_3)I$
and hence
\[
 C(P^I)=\sqrt{4+4\alpha_1\alpha_3}\leq\sqrt{4+2(\alpha_1^2+\alpha_3^2)}\leq\sqrt{6}.
\]
This implies that
\[
\sup_{I\in\s}\{C(P^I)\}\leq\sqrt{6}.
\]
For every $d\geq11$ it is easy to verify that $\sqrt{6}<C(P')$ so
\[
\sup_{I\in\s}\{C(P^I)\}<C(P'),
\]
as announced.
\end{remark}

\begin{remark}
Some of the results presented here can be generalized to real alternative *-algebras, a setting in which polynomials can be defined and share many of the properties valid on the quaternions (see \cite{GhPe_AIM}).
The polynomials given in Corollary \ref{counterexample} can be defined every time the algebra contains an Hamiltonian triple $i,j,k$. This property is equivalent to say that the algebra contains $\hh$ as a subalgebra (see \cite[\S8.1]{Numbers}). For example, this is true for the algebra of octonions and for the Clifford algebras with signature $(0,n)$, with $n\ge2$. Therefore in all such algebras there exist polynomials for which the zero set $V(P')$ (as a subset of the \emph{quadratic cone}) is not included in the circularization of the convex hull of $V(N(P))$ viewed as a complex polynomial.
\end{remark}






\end{document}